\documentclass[12pt, reqno]{amsart}
\usepackage{amsmath, amsthm, amscd, amsfonts, amssymb, graphicx, color}
\usepackage[bookmarksnumbered, colorlinks, plainpages]{hyperref}
\hypersetup{colorlinks=true,linkcolor=red, anchorcolor=green, citecolor=cyan, urlcolor=red, filecolor=magenta, pdftoolbar=true}

\newtheorem{theorem}{Theorem}[section]
\newtheorem{lemma}[theorem]{Lemma}
\newtheorem{prop}[theorem]{Proposition}
\newtheorem{cor}[theorem]{Corollary}
\theoremstyle{definition}

\theoremstyle{remark}
\newtheorem{remark}[theorem]{\bf{Remark}}
\numberwithin{equation}{section}
\newcommand{\vertiii}[1]{{\left\vert\kern-0.25ex\left\vert\kern-0.25ex\left\vert #1 
    \right\vert\kern-0.25ex\right\vert\kern-0.25ex\right\vert}}

\allowdisplaybreaks
\begin{document}

\title[ Norm Inequalities for Hilbert space operators with Applications ]  {  Norm Inequalities for Hilbert space operators with Applications  }


\author[P. Bhunia ]{ Pintu Bhunia }

\address{ Department of Mathematics, Indian Institute of Science, Bengaluru 560012, Karnataka, India}
\email{pintubhunia5206@gmail.com; pintubhunia@iisc.ac.in}


\thanks{ The author would like to sincerely acknowledge Prof. Apoorva Khare for his valuable comments on this article.
The author also would like to thank SERB, Govt. of India for the financial support in the form of National Post Doctoral Fellowship (N-PDF, File No. PDF/2022/000325) under the mentorship of Prof. Apoorva Khare}


\date{}
\maketitle
\begin{abstract}

Several unitarily invariant norm inequalities and numerical radius inequalities for Hilbert space operators are studied. We investigate some necessary and sufficient conditions for the parallelism of two bounded operators. For a finite rank operator $A,$ it is shown that
\begin{eqnarray*}
    \|A\|_{p} &\leq &\left(\textit{rank} \,  A\right)^{1/{2p}} \|A\|_{2p} \,\, \leq \,\, \left(\textit{rank} \,  A\right)^{{(2p-1)}/{2p^2}} \|A\|_{2p^2}, \quad \textit{for all $p\geq 1 $}
\end{eqnarray*} 
 where $\|\cdot\|_p$ is the Schatten $p$-norm. If $\{ \lambda_n(A) \}$ is a listing of all non-zero eigenvalues (with multiplicity) of a compact operator $A$, then we show that
 \begin{eqnarray*}
        \sum_{n} \left|\lambda_n(A)\right|^{p} &\leq& \frac12 \| A\|_{ p}^{ p}  + \frac12 \| A^2\|_{p/2}^{p/2}, \quad \textit{for all $p\geq 2$}
\end{eqnarray*}
 which improves the classical Weyl's inequality $\sum_{n} \left|\lambda_n(A)\right|^{p} \leq \| A\|_{ p}^{ p}$ [Proc. Nat. Acad. Sci. USA 1949]. For an $n\times n$ matrix $A$, we show that the function $p\to n^{-{1}/{p}}\|A\|_p$ is monotone increasing on $p\geq 1,$ complementing the well known decreasing nature of $p\to \|A\|_p.$ \\
 \indent As an application of these inequalities, we provide an upper bound for the sum of the absolute values of the zeros of a complex polynomial. As  another application we provide a refined upper bound for the energy of a graph $G$, namely,
 $\mathcal{E}(G) \leq  \sqrt{2m\left(\textit{rank Adj(G)} \right)},$ where $m$ is the number of edges, improving on a bound by McClelland in $1971$.

\end{abstract}

\tableofcontents

\section{Introduction and notation}

\noindent Suppose $\mathcal{B}(\mathcal{H})$ denotes the $\mathbb{C}^*$-algebra of all bounded linear operators on a complex Hilbert space $\mathcal{H}.$ If $\mathcal{H}$ is an $n$-dimensional space, then $\mathcal{B}(\mathcal{H})$ is identified with $\mathcal{M}_n(\mathbb{C}),$ the set of all $n\times n$ complex matrices.
For $A\in \mathcal{B}(\mathcal{H}),$ let $|A|=(A^*A)^{1/2}$, where $A^*$ denotes the adjoint of $A.$ Let $\Re(A)=\frac{1}{2}(A+A^*)$ and $\Im(A)=\frac{1}{2i}(A-A^*)$ be the real and imaginary parts of $A$, respectively.
For $A\in \mathcal{B}(\mathcal{H}),$ let $r(A)$ and $\|A\|$ denote the spectral radius and the operator norm of $A,$ respectively. The numerical range and the numerical radius of $A$, denoted by $W(A)$ and $w(A)$ respectively, are defined as
$W(A)=\{ \langle Ax,x\rangle : x\in \mathcal{H}, \|x\|=1 \}$
and 
$w(A)=\sup\{ |\lambda| : \lambda \in W(A)\}.$ It is well known that the numerical radius defines a norm on $\mathcal{B}(\mathcal{H})$ and is equivalent to the operator norm. More precisely, it satisfies the following inequalities 
\begin{eqnarray}\label{equation}
    \frac{1}{2}\|A\| \leq \max\left( \frac{1}{2}\|A\|, \, r(A)\right) \leq w(A) \leq  \|A\|.
\end{eqnarray} 
If $A^2=0,$ then $w(A)=\frac{1}{2}\|A\|$; if $A$ is normal, then $r(A)=w(A)=\|A\|.$

Recall that (\cite{Penrose1, Penrose2}) the Moore-Penrose inverse $A^{\dagger}$ of an operator $A\in \mathcal{B}(\mathcal{H})$ with closed range, is the unique operator in $\mathcal{B}(\mathcal{H})$ which satisfies $$ AA^{\dagger}A=A,  \quad A^{\dagger}AA^{\dagger}=A^{\dagger}, \quad (AA^{\dagger})^*=AA^{\dagger}, \quad (A^{\dagger}A)^*=A^{\dagger}A.$$
For every $A\in \mathcal{B}(\mathcal{H})$ with closed range, we have  $(A^{\dagger})^{\dagger}=A$, $(A^{\dagger})^*=(A^*)^{\dagger}$, $AA^{\dagger}=P_{\mathcal{R}(A)}$ and $A^{\dagger}A=P_{\mathcal{R}(A^*)},$ where $P_{\mathcal{R}(A)}$ is the orthogonal projection onto the range space ${\mathcal{R}(A)}$ of $A.$ 

\noindent

Suppose $\mathcal{K}(\mathcal{H})$ denotes the subset of all compact operators in $\mathcal{B}(\mathcal{H})$. Throughout when we talk about compact operators, we will always consider $\mathcal{H}$ to be a separable Hilbert space.
If $A\in \mathcal{K}(\mathcal{H})$, then the singular values of $A$ (the eigenvalues of $|A|$) are denoted as $s_j(A)$ for $j=1,2, \ldots $ with $s_j(A)\geq s_{j+1}(A)$. For $A\in \mathcal{K}(\mathcal{H})$, let
\begin{eqnarray*}
	\|A\|_p= \left( \sum_{j}s_j^p(A)\right)^{1/p}=\left( \textit{trace}\, |A|^p  \right)^{1/p}, \quad \textit{$p>0$}.
\end{eqnarray*}
Then for $p\geq 1$ ($0<p<1$), $\|\cdot\|_p$ defines a norm (quasi-norm) on the $p$-Schatten class $\mathcal{C}_p(\mathcal{H})=\{A\in \mathcal{K}(\mathcal{H}) : \|A\|_p<\infty\}.$ 
This is known as the Schatten $p$-norm. Here
 $\|A\|_{\infty}=\|A\|=s_1(A)$ and $\|A\|_2$ is the Hilbert-Schmidt norm. For $1\leq p\leq q \leq \infty $, the Schatten $p$-norm satisfies the monotonicity property $$\|A\|_{\infty} \leq \|A\|_{q}\leq \|A\|_p\leq \|A\|_{1}.$$
Throughout when we write $\|A\|_p$, we always consider $A\in \mathcal{C}_p(\mathcal{H}),$ for  $p>0.$
 A norm $\vertiii{\cdot} $ defined on a two sided ideal $\mathcal{C}_{\vertiii{\cdot}}(\mathcal{H})$ of $\mathcal{B}(\mathcal{H})$ is said to be unitarily invariant if $\vertiii{UAV}= \vertiii{A} $ for all $A\in \mathcal{C}_{\vertiii{\cdot}}(\mathcal{H})$ and for all unitary operators $U,V\in \mathcal{B}(\mathcal{H}).$ When we talk of $\vertiii{A}$, we are considering $A\in \mathcal{C}_{\vertiii{\cdot}}(\mathcal{H})$ (and $\vertiii{\cdot}$ to be unitarily invariant).
The Schatten $p$-norm and the operator norm are examples of unitarily invariant norms. However, the numerical radius norm is not unitarily invariant, it is a weakly unitarily invariant norm, i.e., $w(U^*AU)=w(A)$ for every $A\in \mathcal{B}(\mathcal{H})$ and for every unitary operator $U\in \mathcal{B}(\mathcal{H}).$ For $A,B\in  \mathcal{B}(\mathcal{H}),$ the direct sum $A\oplus B$ denotes the $2\times 2$ operator matrix $\begin{bmatrix}
A&0\\
0&B
\end{bmatrix}$ and $\|A\oplus B\|=\max (\|A\|, \|B\|)$. For $A,B\in \mathcal{C}_p(\mathcal{H})$, $\|A\oplus B\|_p= \left(\|A\|_p^p+ \|B\|_p^p\right)^{1/p}.$ 
 For $A,B\in \mathcal{C}_{\vertiii{\cdot}}(\mathcal{H})$, the following results hold:
\begin{eqnarray}\label{10}
\vertiii{A \oplus B}= \vertiii{\begin{bmatrix}
	0&A\\
	B&0
	\end{bmatrix}},
\end{eqnarray}
\begin{eqnarray}\label{2}
\vertiii{A }=\vertiii{\,|A| \,} =\vertiii{A^* },
\end{eqnarray}
\begin{eqnarray}\label{3}
\vertiii{A^*A }=\vertiii{AA^* },
\end{eqnarray}
\begin{eqnarray}\label{4}
\vertiii{ \, |A|\, |B|\, }=\vertiii{AB^* },
\end{eqnarray}

and 
\begin{eqnarray}\label{5}
\vertiii{A\oplus A^* }=\vertiii{A\oplus A }= \vertiii{A\oplus |A| \,}.
\end{eqnarray}

\section{Preliminaries and main results}\label{prel}

\noindent In this section, we motivate and present the statements of our main results. The full details can be found below in the paper.\\
In $1990$, Bhatia and Kittaneh \cite{Bhatia1990} developed an operator arithmetic-geometric mean inequality, i.e., for  $A,B \in \mathcal{K}(\mathcal{H}),$
\begin{eqnarray}\label{lem01}
\vertiii{AB^*} \leq \frac12 \vertiii{A^*A+B^*B}.
\end{eqnarray} 

\noindent In $1997$, Kittaneh \cite{kittaneh1997} proved that for positive  operators $X,Y \in \mathcal{K}(\mathcal{H}),$
\begin{eqnarray} \label{p0p}
\vertiii{(X+Y)\oplus 0} &\leq& \vertiii{X\oplus Y}+ \vertiii{ X^{1/2}Y^{1/2} \oplus X^{1/2}Y^{1/2 } }.
\end{eqnarray}
Recently, for $A,B\in \mathcal{M}_n(\mathbb{C})$, the following operator norm (spectral norm) inequalities 
\begin{eqnarray}\label{001p}
\|AB\pm BA\| &\leq& \|A\| \|B\| 
+  \frac12 \|A^*B\pm BA^*\|
\end{eqnarray}
are shown in \cite{kit1, Natoor2022}.

In Section \ref{sec3}, we develop several unitarily invariant norm inequalities for the sums and products of compact operators, which generalize and improve the the existing inequalities \eqref{lem01}--\eqref{001p}.
From these we derive several Schatten $p$-norm and operator norm inequalities. 

 In Section \ref{sec4}, by using the Moore-Penrose inverse of an operator, we develop additional Schatten $p$-norm inequalities. For a finite rank operator $A,$ we show that
 \begin{theorem} (See Corollary \ref{vcor})
     \begin{eqnarray*}
    \|A\|_{p} &\leq &\left(\textit{rank} \,  A\right)^{1/{2p}} \|A\|_{2p} \,\, \leq \,\, \left(\textit{rank} \,  A\right)^{{(2p-1)}/{2p^2}} \|A\|_{2p^2}, \quad \textit{for all $p\geq 1 $.}
\end{eqnarray*} 
 \end{theorem}

To motivate our next contribution, recall that in 1949, Weyl \cite{Weyl} proved that 
 \begin{eqnarray}\label{wel}
     \sum_{n} \left|\lambda_n(A)\right|^{p} \leq \| A\|_{ p}^{ p},  \quad \textit{for all $p\geq 1$}
 \end{eqnarray}
where $\{ \lambda_n(A) \}$ is a listing of all non-zero eigenvalues (with multiplicity) of a compact operator $A$.  In 1977, Simon \cite{B.Simon} provided another proof of Weyl's inequality \eqref{wel}. We provide an improvement of \eqref{wel}, namely,
\begin{theorem} (See Corollary \ref{cor..})
     \begin{eqnarray*}
        \sum_{n} \left|\lambda_n(A)\right|^{p} &\leq& \frac12 \| A\|_{ p}^{ p}  + \frac12 \| A^2\|_{p/2}^{p/2}, \quad \textit{for all $p\geq 2$}.
\end{eqnarray*}
\end{theorem}

In Section \ref{sec5}, we study the operator norm and numerical radius inequalities of bounded linear operators. In $1997$, Kittaneh \cite{kittaneh1997} obtained that for positive operators $X,Y\in \mathcal{B}(\mathcal{H})$,
\begin{eqnarray}\label{97}
    \|X+Y\| &\leq& \max (\|X\|, \|Y\|)+ \|X^{1/2} Y^{1/2}\|.
\end{eqnarray}
Later on, in $2002$,  Kittaneh \cite{JOT2002} improved the inequality \eqref{97} as
 \begin{eqnarray}\label{2002jot}
	\|X+Y\| \leq \frac{ \|X\|+\|Y\| + \sqrt{ (\|X\|-\|Y\|)^2 +4 \| X^{1/2}Y^{1/2} \|  } }{2}.
	\end{eqnarray}
  We obtain a generalization of the inequality \eqref{2002jot} as
  \begin{theorem}(See Theorem \ref{th5})
      \begin{eqnarray*}\label{}
	\|X+Y\| \leq \frac{ \|X\|+\|Y\| + \sqrt{ (\|X\|-\|Y\|)^2 +4 \| X^{1-t}Y^{1-\alpha} \| \| X^{t}Y^{\alpha}\| } }{2}, \quad  \forall \alpha, t\in [0,1] .
	\end{eqnarray*}

  \end{theorem}
 
Various numerical radius inequalities improving the inequalities \eqref{equation}, have been studied in various articles, see \cite{Alomari, Bhunia2, withkittaneh, FAA, Book2022,   Feki, GUS1997, Kittaneh_LAMA_2023} and the references therein. One of the most well known refinements of the upper bound in \eqref{equation} is 
\begin{eqnarray}\label{k03}
w(A) \leq \frac12 \|A\|+ \frac12 \sqrt{\|A^2\|},
\end{eqnarray}
which is given in \cite{Kittaneh_2003}. An improvement of the inequality \eqref{k03} is given in \cite{Bhunia2021}, namely,
\begin{eqnarray}\label{pintu22}
	w(A) \leq \frac12 \|A\|+ \frac12 {r^{1/2} \left( |A|^{}|A^*|^{}\right )}.
	\end{eqnarray}
In recent work \cite[Th. 2.20]{Bhunia2024}, we developed the more refined and generalized upper bound 
\begin{eqnarray}
    w(A) \leq \frac12 \|A\|+ \frac12 \|A\|^t \sqrt{r\left(|A|^{1-t}|A^*|^{1-t}\right)}, \quad \forall t\in [0,1]. 
\end{eqnarray}
Below we now show that
\begin{theorem}(See Theorem \ref{corp11})
    \begin{eqnarray*}
		w(A) &\leq& \frac12 \|A\|+ \frac12 r^{1/4} \left( |A|^{2t}|A^*|^{2\alpha}\right )   r^{1/4}\left ( |A|^{2(1-t)}|A^*|^{2(1-\alpha)}\right),\quad {\forall \alpha, t\in [0,1].}
	\end{eqnarray*}
\end{theorem}
\noindent This result improves and generalizes the inequalities \eqref{k03} and \eqref{pintu22}. We give a sufficient condition for the norm equality. As an application of the operator norm inequalities, we study the necessary and sufficient conditions for the parallelism of two bounded linear operators.
Recall that (see \cite{Zamani}) an operator $A \in \mathcal{B}(\mathcal{H})$ is said to be parallel to $B\in \mathcal{B}(\mathcal{H})$, denoted as $A\parallel B$, if there exists a scalar $\lambda$ with $|\lambda|=1$ such that $\|A+\lambda B\|=\|A\|+\|B\|.$ 

In Section \ref{sec6}, we consider $p(z) = z^n + a_nz^{n-1} + \ldots + a_2z + a_1 $, a complex polynomial of degree $n\geq 2 $ with $a_1 \neq 0$. By applying the Schatten $p$-norm inequalities,  we give an estimation for the sum of the absolute values of the zeros of $p(z)$.   In particular, we show that
 if $\lambda_1,\lambda_2,\ldots, \lambda_n$ are the zeros of  $p(z),$ then
 \begin{theorem}(See Theorem \ref{poly})
      \begin{eqnarray*}
      \sum_{i=1}^n|\lambda_i| &\leq& \sqrt{n\left(n-1 + \sum_{i=1}^n|a_i|^2 \right)}.
    \end{eqnarray*}
 \end{theorem}

In Section \ref{sec7}, we consider a simple graph $G$ with $n$ vertices and $m $ edges.
The energy of the graph, $\mathcal{E} (G)$ is the sum of the absolute values of the eigenvalues of the adjacency matrix Adj$(G).$  In 1971, McClelland \cite{Energy} provided an upper bound for the energy:
$\mathcal{E} (G) \leq  \sqrt{2mn}.$
Employing the Schatten $p$-norm inequalities, we improve on this, by showing:
\begin{theorem}(See Theorem \ref{graph1})
   $$
    \mathcal{E} (G) \leq  \sqrt{2m \left(\text {rank Adj$(G)$ }\right)}.
$$
\end{theorem}

\section{Unitarily invariant norm inequalities }\label{sec3}

In this section, we develop unitarily invariant norm inequalities for the sums and products of operators, and derive several Schatten $p$-norm  inequalities and operator norm inequalities. We begin with the following theorem.



\begin{theorem}\label{th1}
	If $A,B,X,Y\in \mathcal{K}(\mathcal{H})$, then\\
	\begin{eqnarray*}
		\vertiii{(AXB+BYA) \oplus 0} 
		&\leq& \frac12 \vertiii{(|A|^2+|B^*X^*|^2)\oplus (|B|^2+|A^*Y^*|^2) } \\
		&& + \frac12 \vertiii{(A^*B+XBA^*Y^*)\oplus (A^*B+XBA^*Y^*) }
	\end{eqnarray*}
	and 
	\begin{eqnarray*}
		\vertiii{(AXB+BYA) \oplus 0} &\leq& \frac12 \vertiii{(|AX|^2+|B^*|^2)\oplus (|A^*Y^*|^2+|B|^2) } \\
		&& + \frac12 \vertiii{(X^*A^*B+BA^*Y^*)\oplus (X^*A^*B+BA^*Y^*) }.
	\end{eqnarray*}
	
\end{theorem}
\begin{proof}
	We have
	\begin{eqnarray*}
		&& \vertiii{(AXB+BYA) \oplus 0} \\
		&= & \vertiii{\begin{bmatrix}
				A&B\\
				0&0
			\end{bmatrix} \begin{bmatrix}
				XB&0\\
				YA&0
		\end{bmatrix} }\\
		&\leq& \frac12 \vertiii{\begin{bmatrix}
				A^*&0\\
				B^*&0
			\end{bmatrix} \begin{bmatrix}
				A&B\\
				0&0
			\end{bmatrix}  +  \begin{bmatrix}
				XB&0\\
				YA&0
			\end{bmatrix} \begin{bmatrix}
				B^*X^*&A^*Y^*\\
				0&0
		\end{bmatrix}  } \,\, (\textit{using \eqref{lem01}})\\
		&=& \frac12 \vertiii{\begin{bmatrix}
				|A|^2&A^*B\\
				B^*A&|B|^2
			\end{bmatrix} +  \begin{bmatrix}
				|B^*X^*|^2&XBA^*Y^*\\
				YAB^*X^*&|A^*Y^*|^2
		\end{bmatrix}  }\\
		&=& \frac12 \vertiii{\begin{bmatrix}
				|A|^2+|B^*X^*|^2&0\\
				0&|B|^2+|A^*Y^*|^2
			\end{bmatrix} +  \begin{bmatrix}
				0&A^*B+XBA^*Y^*\\
				B^*A+YAB^*X^*&0
		\end{bmatrix}  }\\
		&\leq & \frac12 \vertiii{\begin{bmatrix}
				|A|^2+|B^*X^*|^2&0\\
				0&|B|^2+|A^*Y^*|^2
		\end{bmatrix} } \\
		&& + \frac12 \vertiii{ \begin{bmatrix}
				0&A^*B+XBA^*Y^*\\
				B^*A+YAB^*X^*&0
		\end{bmatrix}  }\\
		&= & \frac12 \vertiii{
			(|A|^2+|B^*X^*|^2) \oplus (|B|^2+|A^*Y^*|^2)}  \\
		&& + \frac12 \vertiii{ 
			(A^*B+XBA^*Y^*)\oplus 
			(B^*A+YAB^*X^*) } \,\, (\textit{using \eqref{10}}) \\
		&= & \frac12 \vertiii{
			(|A|^2+|B^*X^*|^2) \oplus (|B|^2+|A^*Y^*|^2)} \\
		&&+ \frac12 \vertiii{ 
			(A^*B+XBA^*Y^*)\oplus 
			(A^*B+XBA^*Y^*) }, \,\, (\textit{using \eqref{5}}).
	\end{eqnarray*}
	Again, using the fact $ \vertiii{(AXB+BYA)\oplus 0}=
	\vertiii{\begin{bmatrix}
		AX&B\\
		0&0
		\end{bmatrix} \begin{bmatrix}
		B&0\\
		YA&0
		\end{bmatrix} }$
  and following similar techniques we get the desired second inequality.
\end{proof}

Considering $X=Y=\pm I$ ($I$ denotes the identity operator) in Theorem \ref{th1}, yields:

\begin{cor}\label{cor1}
	If $A,B\in \mathcal{K}(\mathcal{H})$, then\\
	\begin{eqnarray}\label{pp01}
	\vertiii{(AB\pm BA) \oplus 0} &\leq& \frac12 \vertiii{(A^*A+BB^*)\oplus (AA^*+B^*B) } \notag \\
	&& + \frac12 \vertiii{(A^*B\pm BA^*)\oplus (A^*B\pm BA^*) }.
	\end{eqnarray}

	
\end{cor}

In particular, if we consider the operator norm and  Schatten $p$-norm respectively in Corollary \ref{cor1}, we get the following inequalities for the commutators and anti-commutators of operators.

\begin{cor}\label{cor02}
(i) If $A,B\in \mathcal{K}(\mathcal{H})$, then
	\begin{eqnarray}\label{p1}
	\|AB\pm BA\| &\leq& \frac12 \max \left(\|A^*A+BB^*\|, \|AA^*+B^*B\| \right)  
	+  \frac12 \|A^*B\pm BA^*\|.
	\end{eqnarray}
	

	\noindent (ii) If $A,B\in \mathcal{C}_p(\mathcal{H})$, $1\leq p<\infty$, then
	\begin{eqnarray*}
		\|AB\pm BA\|_p &\leq& \frac12\left(\|A^*A+BB^*\|_p^p+ \|AA^*+B^*B\|_p^p\right)^{1/p} 
		+  2^{\frac{1-p}{p}}\|A^*B\pm BA^*\|_p.
	\end{eqnarray*}
	
	 
\end{cor}

Next we derive the existing inequalities \eqref{001p}  from  the inequalities \eqref{p1}.

\begin{prop}\label{remp}
	If $A,B\in \mathcal{K}(\mathcal{H})$, then 
	\begin{eqnarray}\label{001}
	\|AB\pm BA\| &\leq& \|A\| \|B\| 
	+  \frac12 \|A^*B\pm BA^*\|.
	\end{eqnarray}

\end{prop}
\begin{proof}
	Following the inequalities in \eqref{p1}, we get
	\begin{eqnarray*}
		\|AB\pm BA\| &\leq& \frac12 \left(\|A\|^2+ \|B\|^2 \right) 
		+ \frac12 \|A^*B\pm BA^*\|.
	\end{eqnarray*}
	Replacing $A$ by $tA$ and $B$ by $\frac1t B$, $t>0$, we get
	\begin{eqnarray}\label{p10}
	\|AB\pm BA\| &\leq& \frac12 \left(t^2 \|A\|^2+ \frac1{t^2}\|B\|^2 \right)  
	+  \frac12 \|A^*B\pm BA^*\|.
	\end{eqnarray}
	Since \eqref{p10} holds for all $t>0,$  considering $t=\sqrt{\frac{\|B\|}{\|A\|}}$ (when $A\neq 0$ and $B\neq 0$), we get
	\begin{eqnarray*}
		\|AB\pm BA\| &\leq& \|A\| \|B\| 
		+  \frac12 \|A^*B\pm BA^*\|.
	\end{eqnarray*}
	The inequalities in \eqref{001} also hold trivially when $\min (\|A\|, \|B\|)=0.$
	\end{proof}

We now obtain the following stronger inequalities than the existing ones in \eqref{001p}.

\begin{cor}\label{corppppp}
If $A,B\in \mathcal{K}(\mathcal{H})$ are non-zero, then
\begin{eqnarray}\label{p00001}
\|AB\pm BA\| &\leq& \frac12 \max \left(   \left\|\frac{\|B\|}{\|A\|}A^*A+\frac{\|A\|}{\|B\|}BB^* \right \|, \left\|\frac{\|B\|}{\|A\|} AA^*+\frac{\|A\|}{\|B\|}B^*B \right\| \right)\notag\\   
&& +   \frac12 \|A^*B\pm BA^*\|.
\end{eqnarray}


\end{cor}

\begin{proof}
Replacing $A$ by $tA$ and $B$ by $\frac1t B$  ($t>0)$ in \eqref{p1}, we get
\begin{eqnarray*}\label{}
\|AB\pm BA\| &\leq& \frac12 \max \left(\|t^2A^*A+\frac{1}{t^2}BB^*\|, \|t^2AA^*+\frac{1}{t^2}B^*B\| \right)   
+   \frac12 \|A^*B\pm BA^*\|.
\end{eqnarray*}
Since these hold for all $t>0,$ we set $t=\sqrt{\frac{\|B\|}{\|A\|}}$ and obtain the desired inequalities.
\end{proof}

 
As a consequence of Corollary \ref{cor1}, we obtain the following results.

\begin{cor}\label{cor2}
	If $A\in \mathcal{K}(\mathcal{H})$, then
	\begin{eqnarray}\label{00p}
	\vertiii{A^2 \oplus 0} &\leq& \frac12 \vertiii{(A^*A+AA^*)\oplus (A^*A+AA^*) }, 
	\end{eqnarray} 
	\begin{eqnarray}\label{01p}
	\vertiii{(A^*A+AA^*)\oplus 0} &\leq& \vertiii{A^*A \oplus AA^*} 
	+  \vertiii{ A^2\oplus A^2 }
	\end{eqnarray}
	and 
	\begin{eqnarray}\label{02p}
	\vertiii{(A^*A-AA^*)\oplus 0} &\leq& \vertiii{A^*A \oplus AA^*}.
	\end{eqnarray}
\end{cor}
\begin{proof}
	The inequalities \eqref{00p} and \eqref{01p} follow from \eqref{pp01} by considering $B=A$ and $B=A^*$, respectively. The inequalitity \eqref{02p} follows from \eqref{pp01} by taking $B=A^*.$
\end{proof}

 Note that the inequality \eqref{01p} is also proved in \cite{kittaneh2002} by employing the polar decomposition of operators. 
Considering the operator norm and the Schatten $p$-norm respectively in Corollary \ref{cor2}, we get the following results. 

\begin{remark}
(i) If $A\in \mathcal{K}(\mathcal{H})$, then
	\begin{eqnarray}\label{pm}
	2\|A^2\| \leq \|A^*A+AA^*\|\leq \|A^2\|+\|A\|^2
	\end{eqnarray}

	and
	$$\|A^*A-AA^*\| \leq \|A\|^2.$$

 The inequality \eqref{pm} is also proved in \cite{kittaneh2002}.

	\noindent (ii) If $A\in \mathcal{C}_p(\mathcal{H})$, $1\leq p< \infty,$ then
	\begin{eqnarray*}
		\|A^2\|_p &\leq& 2^{\frac{1-p}p}\|A^*A+AA^* \|_p \leq 2^{1/p} \|A\|_{2p}^2,
	\end{eqnarray*}
	\begin{eqnarray*}
		\|A^*A+AA^*\|_p &\leq& 2^{1/p} \left( \|A^*A\|_p
		+  \|A^2\|_p\right)= 2^{1/p} \left( \|A\|_{2p}^2
		+  \|A^2\|_p\right)
	\end{eqnarray*}
	
	and 
	\begin{eqnarray*}
		\|A^*A-AA^*\|_p &\leq& 2^{1/p} \|A\|_{2p}^2.
	\end{eqnarray*}

\end{remark}

\noindent Applying Theorem \ref{th1} and using the Moore-Penrose inverse of an operator, we obtain:

\begin{cor}\label{pen1}
	If $A\in \mathcal{K}(\mathcal{H})$ has closed range, then
	\begin{eqnarray}\label{00001}
	\vertiii{A\oplus 0} &\leq& \frac12 \vertiii{ (A^*A+ A^{\dagger}A) \oplus (A^*A+ A^{\dagger}A)  }
	\end{eqnarray}
	
	and 
	
	\begin{eqnarray}\label{00002}
	\vertiii{A\oplus 0} &\leq& \frac14 \vertiii{ (AA^*+ AA^{\dagger}) \oplus (A^*A+ A^{\dagger}A)  }+ \frac12 \vertiii{A\oplus A}.
	\end{eqnarray}
	
\end{cor}

\begin{proof}
	Taking $B=A$ and $X=Y=A^{\dagger}$ in the first inequality of Theorem \ref{th1} and since $A=AA^{\dagger}A,$ we get
	\begin{eqnarray*}
	2 \vertiii{A\oplus 0} &\leq& \frac12 \vertiii{ \left(|A|^2+ |(A^{\dagger}A)^*|^2\right)    \oplus \left(|A|^2+ |(A^{\dagger}A)^*|^2\right)}\\
	 &&+  \frac12 \vertiii{ \left(|A|^2+ A^{\dagger}A (A^{\dagger}A)^*\right)    \oplus \left(|A|^2+ A^{\dagger}A(A^{\dagger}A)^*\right)}\\
	 &=& \vertiii{ (A^*A+ A^{\dagger}A) \oplus (A^*A+ A^{\dagger}A)  },
\end{eqnarray*} 
	which gives the desired inequality \eqref{00001}. Similarly, from the second inequality in Theorem \ref{th1}, we obtain the desired inequality \eqref{00002}.
\end{proof}

In the following, we obtain an unitarily invariant norm inequality for the sum of positive operators, which generalizes the existing inequality \eqref{p0p}.

\begin{theorem}\label{th2}
	If $X,Y\in \mathcal{K}(\mathcal{H})$ are positive, then
	\begin{eqnarray*}
		\vertiii{(X+Y)\oplus 0} &\leq& \frac12 \vertiii{X^{2t}\oplus Y^{2\alpha}}+ \frac12 \vertiii{X^{2(1-t)}\oplus Y^{2(1-\alpha})} \\
		&& +\frac12 \vertiii{(X^tY^{\alpha}+  X^{1-t}Y^{1-\alpha})\oplus (X^tY^{\alpha}+  X^{1-t}Y^{1-\alpha}) },
	\end{eqnarray*}
	
	for all $\alpha, t\in [0,1].$ 
	In particular, for $\alpha=t=\frac12$
	\begin{eqnarray} \label{p0101}
	\vertiii{(X+Y)\oplus 0} &\leq& \vertiii{X\oplus Y}+ \vertiii{ X^{1/2}Y^{1/2} \oplus X^{1/2}Y^{1/2 } }.
	\end{eqnarray}
\end{theorem}

\begin{proof}
	First we write $X+Y= X^tX^{1-t}+ Y^{\alpha}Y^{1-\alpha}.$ Therefore, we have
	\begin{eqnarray*}
		&& \vertiii{(X+Y)\oplus 0}\\
		&=& \vertiii{ \begin{bmatrix}
				X^t & Y^{\alpha}\\
				0&0
			\end{bmatrix} \begin{bmatrix}
				X^{1-t} &0\\
				Y^{1-\alpha} &0
		\end{bmatrix}}\\
		&\leq& \frac12 \vertiii{ \begin{bmatrix}
				X^t & 0\\
				Y^{\alpha} &0
			\end{bmatrix} \begin{bmatrix}
				X^{t} &Y^{\alpha}\\
				0 &0
			\end{bmatrix} + \begin{bmatrix}
				X^{1-t} &0\\
				Y^{1-\alpha} &0
			\end{bmatrix} \begin{bmatrix}
				X^{1-t} & Y^{1-\alpha}\\
				0&0
		\end{bmatrix} }  \,\, (\text{using \ref{lem01}}) \\
		&=& \frac12 \vertiii{ \begin{bmatrix}
				X^{2t}+X^{2(1-t)} &X^tY^{\alpha}+X^{1-t}Y^{1-\alpha}\\
				Y^{\alpha} X^t+Y^{1-\alpha}X^{1-t} & Y^{2\alpha}+Y^{2(1-\alpha)} 
		\end{bmatrix} }\\
		&\leq& \frac12 \vertiii{ (X^{2t}+X^{2(1-t)}) \oplus (Y^{2\alpha}+Y^{2(1-\alpha))}}\\
		&&+ \frac12 \vertiii{ (X^tY^{\alpha}+X^{1-t}Y^{1-\alpha}) \oplus (Y^{\alpha} X^t+Y^{1-\alpha}X^{1-t})}\\
		&\leq&  \frac12 \vertiii{X^{2t}\oplus Y^{2\alpha}}+ \frac12 \vertiii{X^{2(1-t)}\oplus Y^{2(1-\alpha})} \\
		&& +\frac12 \vertiii{(X^tY^{\alpha}+  X^{1-t}Y^{1-\alpha})\oplus (X^tY^{\alpha}+  X^{1-t}Y^{1-\alpha}) } \,\, (\text{using \eqref{5}}) \qedhere. 
	\end{eqnarray*}
\end{proof}




	
	

To prove our next result we need the following lemma.

\begin{lemma}\cite{kittaneh1997}\label{lem1111}
	Let $A,B\in \mathcal{B}(\mathcal{H}).$ If $AB$ is selfadjoint, then
	$\vertiii{AB}\leq \vertiii{\Re(BA)}.$
\end{lemma}

\begin{theorem}\label{th3}
	Let $A,B,X,Y\in \mathcal{K}(\mathcal{H})$ and let $X,Y$ be positive. Then
	\begin{eqnarray*}
		\vertiii{(AXB+BYA) \oplus 0} &\leq& \frac12 \vertiii{\Re (X(|A|^2+|B^*|^2))\oplus \Re(Y(|B|^2+|A^*|^2)) } \\
		&& + \frac12 \vertiii{X^{1/2}(A^*B+BA^*)Y^{1/2}\oplus X^{1/2}(A^*B+BA^*)Y^{1/2} }.
	\end{eqnarray*}

 In particular, for $Y=0$,
 \begin{eqnarray}\label{oppo}
		\vertiii{AXB \oplus 0} &\leq& \frac12 \vertiii{\Re (X(A^*A+BB^*))\oplus 0 }.
	\end{eqnarray}
\end{theorem}

\begin{proof}
	We have
	\begin{eqnarray*}
		&& \vertiii{ (AXB+BYA)\oplus 0} \\
		&=& \vertiii{ \begin{bmatrix}
				AX^{1/2} & BY^{1/2}\\
				0&0
			\end{bmatrix}  \begin{bmatrix}
				X^{1/2}B&0\\
				Y^{1/2}A&0
		\end{bmatrix}} \\
		&\leq& \frac12 \vertiii{ \begin{bmatrix}
				X^{1/2}A^*&0\\
				Y^{1/2}B^*&0
			\end{bmatrix} \begin{bmatrix}
				AX^{1/2} & BY^{1/2}\\
				0&0
			\end{bmatrix}  + \begin{bmatrix}
				X^{1/2}B & 0\\
				Y^{1/2}A&0
			\end{bmatrix}    \begin{bmatrix}
				B^*X^{1/2}&A^*Y^{1/2}\\
				0&0
		\end{bmatrix} } \\
		&& \,\,\,\,\,\,\,\,\,\,\,\,\,\,\,\,\,\,\,\,\,\,\, (\textit{using \eqref{lem01}}) \\
		&= & \frac12 \vertiii{ \begin{bmatrix}
				X^{1/2}(|A|^2+|B^*|^2)X^{1/2}& X^{1/2}(A^*B+BA^*)Y^{1/2}\\
				Y^{1/2}(B^*A+AB^*)X^{1/2}& Y^{1/2}(|B|^2+|A^*|^2)Y^{1/2}
		\end{bmatrix} }\\
		&\leq & \frac12 \vertiii{ X^{1/2}(|A|^2+|B^*|^2)X^{1/2} \oplus Y^{1/2}(|B|^2+|A^*|^2)Y^{1/2} }\\
		&&+ \frac12 \vertiii{ X^{1/2}(A^*B+BA^*)Y^{1/2} \oplus Y^{1/2}(B^*A+AB^*)X^{1/2} }\\
		&\leq & \frac12 \vertiii{ \Re(X(|A|^2+|B^*|^2)) \oplus \Re(Y(|B|^2+|A^*|^2)) }\\
		&&+ \frac12 \vertiii{ X^{1/2}(A^*B+BA^*)Y^{1/2} \oplus X^{1/2}(A^*B+BA^*)Y^{1/2} } \,\, (\textit{by Lemma \ref{lem1111}}),
	\end{eqnarray*}
	
	as desired.
\end{proof}

As a consequence of \eqref{oppo}, we get the following results.

\begin{cor}\label{cor4}
	Let $A,B,X\in \mathcal{K}(\mathcal{H})$ and let $X$ be positive. Then
	for $1\leq p <\infty$, 
	\begin{eqnarray*}
		\|AXB\|_p &\leq& \frac12 \| \Re (X(A^*A+BB^*)) \|_p
	\end{eqnarray*}
	
	and 
	\begin{eqnarray}\label{pp01p}
	\|{AXB }\| &\leq& \frac12 \|{\Re (X(A^*A+BB^*))  }\|.
	\end{eqnarray}
\end{cor}

\begin{remark}
	Let $A,B,X\in \mathcal{K}(\mathcal{H})$ and let $X$ be positive.
	
(i)	From the inequality \eqref{pp01p}, we get 
	\begin{eqnarray}\label{1p1}
	\|{AXB }\| &\leq & \frac12 w({ X(A^*A+BB^*) }).
	\end{eqnarray}

(ii)	Replacing $A$ and $B$ by $\sqrt{t}A$ and $\frac{1}{\sqrt{t}}B$ ($t>0$) respectively in \eqref{1p1}, we obtain
	\begin{eqnarray}\label{1p2}
	\|{AXB }\|^2 &\leq & w(XA^*A) \,\,  w(XBB^*).
	\end{eqnarray}
	
(iii)	If $A$ is positive, then from the inequality \eqref{1p2}, we get
	\begin{eqnarray}\label{1p3}
	\left\| A^{1/2} X^{1/2} \right\| &\leq & w^{1/2}(AX), 
	\end{eqnarray}
	
	which improves the existing known inequality $\left\| A^{1/2} X^{1/2} \right\| \leq \|AX\|^{1/2}$.

(iv)	If $A, B$ are positive, then from the inequality \eqref{1p2}, we obtain 
	\begin{eqnarray}\label{1p4}
	\left\| A^{1/2} X B^{1/2}\right\| &\leq& w^{1/2}(AX) \,\, w^{1/2}(XB).
	\end{eqnarray}
 
  From \eqref{1p1}, we also have
\begin{eqnarray}\label{1p1op}
	\|{A^{1/2}XB^{1/2} }\| &\leq & \frac12 w({ XA+XB }).
	\end{eqnarray}


	
\end{remark}

\noindent Now using the Moore-Penrose inverse we deduce unitarily invariant norm inequalities.

\begin{cor}
		If $A\in \mathcal{K}(\mathcal{H})$ has closed range, then\\
	\begin{eqnarray}\label{9p}
		\vertiii{A \oplus 0} &\leq& \frac12 \vertiii{ (A^*A+A^{\dagger}A)\oplus 0 }
	\end{eqnarray}
	
and
\begin{eqnarray}\label{10p}
	\vertiii{A \oplus 0} &\leq& \frac12 \vertiii{ (AA^*+AA^{\dagger})\oplus 0 }.
\end{eqnarray}
	
\end{cor}

\begin{proof}
	By taking $X=A^{\dagger}A$ and $B=I$ in the first inequality of Corollary \ref{cor4} and since $A=AA^{\dagger}A$, we get
	\begin{eqnarray*}
		\vertiii{A \oplus 0} &\leq& \frac12 \vertiii{ Re(A^{\dagger}A A^*A+A^{\dagger}A)\oplus 0 }\\
		&=& \frac12 \vertiii{ Re( A^*A+A^{\dagger}A)\oplus 0 }\\
		&=& \frac12 \vertiii{ ( A^*A+A^{\dagger}A)\oplus 0 }.\\
	\end{eqnarray*}
This gives \eqref{9p}. The inequality \eqref{10p} follows from \eqref{9p} by replacing $A$ by $A^*.$ 
\end{proof}

Considering the Schatten $p$-norm (for $1\leq p \leq \infty$) in \eqref{9p} and  \eqref{10p} respectively,  we get
	\begin{eqnarray*}
		\|A\|_p &\leq& \frac12 \|A^*A+A^{\dagger}A \|_p=\|A^*A+ P_{\mathcal{R}(A^*)}\|_p
	\end{eqnarray*}

and 
\begin{eqnarray*}
	\|A\|_p &\leq& \frac12 \|AA^*+AA^{\dagger} \|_p=\|AA^*+ P_{\mathcal{R}(A)}\|_p.
\end{eqnarray*}
Therefore, combining the above two inequalities yields the following result.
\begin{prop}\label{propp1}
	If $A\in \mathcal{K}(\mathcal{H})$ has closed range, then
	$$\|A\|_p \leq \frac12 \min\left( \|A^*A+ P_{\mathcal{R}(A^*)}\|_p,   \|AA^*+  P_{\mathcal{R}(A)}\|_p\right),\quad \textit{for all $1\leq p \leq \infty.$} $$

\end{prop}

\section{Schatten $p$-norm inequalities via orthonormal sets}\label{sec4}

We begin this section by
noting (see \cite{Simon}) that  for $A\in \mathcal{C}_p(\mathcal{H}),$ $p\geq 1$,
$$\|A\|_p= \sup \left( \sum_{k} |\langle Ax_k, y_k\rangle|^p\right)^{1/p},$$
where the supremum is taken over all orthonormal sets 
$\{x_k \}$ and $\{ y_k\}$ in $\mathcal{H}$.
First, we obtain a reverse type inequality for the Schatten $p$-norm inequality $\|A\|_{q}\leq \|A\|_p,$ for $1\leq p\leq q .$

\begin{theorem}\label{th100p}
If $A\in \mathcal{B}(\mathcal{H})$ with finite rank, then
\begin{eqnarray*}
\|A\|_{2r} &\leq& \left(\textit{rank} \,  A \right)^{1/{2q}} \|A\|_{2p},
\end{eqnarray*}

where  $1 \leq p,q <\infty$ and $\frac1r= \frac1p+\frac1q.$ 
In particular, for $p=q$,
\begin{eqnarray}\label{ppppp}
\|A\|_{p} &\leq& \left(\textit{rank} \,  A \right)^{1/{2p}} \|A\|_{2p}.
\end{eqnarray}
The inequality \eqref{ppppp} is sharp, and equality holds if rank $A=1$.

\end{theorem}

\begin{proof}
For all $x, y \in \mathcal{H}$, we get
\begin{eqnarray*}
| \langle Ax,y\rangle | &\leq& | \langle AA^{\dagger}Ax,y \rangle | = | \langle P_{\mathcal{R}(A)}Ax,y\rangle | \leq \|Ax\| \| P_{\mathcal{R}(A)}y\|.
\end{eqnarray*}
Therefore, 
\begin{eqnarray}\label{0p0p1}
| \langle Ax,y\rangle |^2 &\leq&  \langle |A|^2 x,x\rangle     \langle P_{\mathcal{R}(A)}y,y\rangle.
\end{eqnarray}
Let $\{x_k \}$ and $\{ y_k\}$   be any two orthonormal sets in $\mathcal{H}$. Then, from \eqref{0p0p1}, we get
\begin{eqnarray*}
| \langle Ax_k,y_k\rangle |^{2r} &\leq&  \langle |A|^2 x_k,x_k\rangle^r     \langle P_{\mathcal{R}(A)}y_k,y_k\rangle^r. 
\end{eqnarray*}
Summing over $k$, we get
\begin{eqnarray*}
\sum_{k} | \langle Ax_k,y_k\rangle |^{2r} &\leq&  \sum_{k}\langle |A|^2 x_k,x_k\rangle^r     \langle P_{\mathcal{R}(A)}y_k,y_k\rangle^r\\
&\leq&   \left(\sum_{k} \left( \langle |A|^2 x_k,x_k\rangle^r \right)^{p/r}  \right)^{r/p}   \left(\sum_{k} \left( \langle   P_{\mathcal{R}(A)} y_k,y_k\rangle^r \right)^{q/r}  \right)^{r/q} \\
&& \,\,\,\,\,\,\,\,\,\,\,\,\,\,\,\,\,\,\,\,\,\,\,\,\,\, (\textit{by H\"{o}lder's inequality})\\
&\leq&   \left(\sum_{k}  \langle |A|^2 x_k,x_k\rangle^p  \right)^{r/p}   \left(\sum_{k}  \langle       P_{\mathcal{R}(A)} y_k,y_k\rangle^q  \right)^{r/q} \\
&\leq&   \left(\sum_{k} \langle |A|^{2p} x_k,x_k\rangle  \right)^{r/p}   \left(\sum_{k}  \langle       P_{\mathcal{R}(A)} y_k,y_k\rangle  \right)^{r/q} \\
&\leq& \left( \textit{trace} \, |A|^{2p}\right)^{r/p}   \left( \textit{trace} \,  P_{\mathcal{R}(A)} \right)^{r/q} \\
&=& \|A\|_{2p}^{2r}\,   (\textit{rank} \,  A )^{r/{q}}.
\end{eqnarray*}

\noindent Therefore, taking the supremum over all orthonormal sets $\{x_k \}$ and $\{y_k\}$ in $\mathcal{H}$, we get
 \begin{eqnarray*}
\|A\|_{2r}^{2r}=\sup \sum_{k} | \langle Ax_k,y_k\rangle |^{2r} &\leq& (\textit{rank} \,  A )^{r/{q}}   \|A\|_{2p}^{2r}. 
\end{eqnarray*}

Considering $p=q$, we get $
\|A\|_{p} \leq (\textit{rank} \,  A )^{1/{2p}} \|A\|_{2p}.$
This completes the proof.
\end{proof}

Again, using similar arguments as Theorem \ref{th100p}, we obtain the following results.

\begin{theorem}\label{th100p1}
If $A\in \mathcal{B}(\mathcal{H})$ with finite rank, then
\begin{eqnarray*}
\|A\|_{2pr} &\leq& \left(\textit{rank} \,  A \right)^{1/{2pq}} \|A\|_{2p^2},
\end{eqnarray*}

where  $1 \leq p,q <\infty$ and $\frac1r= \frac1p+\frac1q.$ 
In particular, for $r=1$,
\begin{eqnarray}\label{ppppp0}
\|A\|_{2p} &\leq& \left(\textit{rank} \,  A \right)^{{(p-1)}/{2p^2}} \|A\|_{2p^2}.
\end{eqnarray}

The inequality \eqref{ppppp0} is sharp, and equality holds if rank $A=1$.

\end{theorem}

Combining the inequalities \eqref{ppppp} and \eqref{ppppp0}, we get the following corollary.

\begin{cor}\label{vcor}
If $A\in \mathcal{B}(\mathcal{H})$ with finite rank, then
    \begin{eqnarray}
    \|A\|_{p} \leq \left(\textit{rank} \,  A\right)^{1/{2p}} \|A\|_{2p} \leq \left( \textit{rank} \,  A\right)^{{(2p-1)}/{2p^2}} \|A\|_{2p^2}, \quad \textit{\textit{for all $p\geq 1$}.}
    \end{eqnarray}

\end{cor}

To prove our next Schatten $p$-norm inequality we need the following lemmas. 

\begin{lemma}\label{lemkitt} \cite{Kittaneh1988} (\textit{Mixed Schwarz inequality}) 
    Let $A\in \mathcal{B}(\mathcal{H})$ and let $f,g$ be non-negative continuous functions on $[0,\infty)$ such that $f(t)g(t)=t, $ for all $t\in [0,\infty).$ Then
    \begin{eqnarray*}
        |\langle Ax,y\rangle|^2 &\leq& \langle f^2(|A|)x,x\rangle \langle g^2(|A^*|)y,y\rangle,\quad \textit{for all $x,y\in \mathcal{H}.$}
    \end{eqnarray*}

\end{lemma}


\begin{lemma}\label{lemcarthy}\cite{Simon} (\textit{McCarthy inequality})
    If $A\in \mathcal{B}(\mathcal{H})$ is positive, then
    $$\langle Ax,x\rangle^p\leq \langle A^px,x\rangle,\quad \textit{for all $x\in \mathcal{H}$ with $\|x\|=1$ and for all $p\geq 1.$} $$

\end{lemma}


\begin{lemma}\cite{Buzano}(\textit{Buzano's inequality}) \label{buz}
    If $x,y,e \in \mathcal{H}$ with $\|e\|=1,$ then
    $$ |\langle x,e\rangle \langle e,y\rangle|\leq \frac12 (\|x\| \|y\|+ |\langle x,y\rangle|).$$
\end{lemma}

We can now prove the following theorem.

\begin{theorem}\label{thp100}
    Let $A\in \mathcal{K}(\mathcal{H})$ and let $f$, $g$ be as in Lemma \ref{lemkitt}. For any orthonormal set $\{x_n\} \subset \mathcal{H},$ 
    \begin{eqnarray*}
        \sum_{n}|\langle Ax_n,x_n\rangle|^{2r} &\leq& \frac12 \| f(|A|)\|_{2p}^{2r} \| g(|A^*|)\|_{2q}^{2r} + \frac12 \| f^2(|A|) g^2(|A^*|)\|_r^r,
    \end{eqnarray*} 
where  $\frac1r=\frac1p+\frac1q$ and $r\geq 1, p\geq 2, q\geq 2.$
In particular, for $p=q,$ 
\begin{eqnarray}\label{0p0p0p1}
        \sum_{n}|\langle Ax_n,x_n\rangle|^{p} &\leq& \frac12 \| f(|A|)\|_{2p}^{p} \| g(|A^*|)\|_{2p}^{p} + \frac12 \| f^2(|A|) g^2(|A^*|)\|_{p/2}^{p/2}.
    \end{eqnarray} 

\end{theorem}

\begin{proof}
     From Lemma \ref{lemkitt} and Lemma \ref{buz}, we have
     \begin{eqnarray*}
         |\langle Ax_n,x_n\rangle|^2&\leq& \langle f^2(|A|)x_n,x_n\rangle \langle g^2(|A^*|)x_n,x_n\rangle\\
         &\leq& \frac12 \left( \|f^2(|A|)x_n\| \|g^2(|A^*|)x_n\|+ |\langle f^2(|A|)g^2(|A^*|)x_n,x_n \rangle| \right).
     \end{eqnarray*}
     Using the convexity of $t^r$, $r\geq 1$, it follows that
      \begin{eqnarray*}
         |\langle Ax_n,x_n\rangle|^{2r} 
         &\leq& 
          \frac12 \left( \|f^2(|A|)x_n\|^r \|g^2(|A^*|)x_n\|^r+ |\langle f^2(|A|)g^2(|A^*|)x_n,x_n \rangle|^r \right)\\
         &=&  \frac12 \left( f^4(|A|)x_n,x_n\rangle^{r/2} \langle g^4(|A^*|)x_n,x_n\rangle^{r/2}+ |\langle f^2(|A|)g^2(|A^*|)x_n,x_n \rangle|^r \right).
     \end{eqnarray*}
Summing over $n$, we get
     \begin{eqnarray*}
         && \sum_n|\langle Ax_n,x_n\rangle|^{2r}\\ 
         & \leq& \frac12 \left( \sum_n f^4(|A|)x_n,x_n\rangle^{r/2} \langle g^4(|A^*|)x_n,x_n\rangle^{r/2}+ \sum_n|\langle f^2(|A|)g^2(|A^*|)x_n,x_n \rangle|^r \right).
     \end{eqnarray*}
Using H\"{o}lder's inequality, we have
     \begin{eqnarray*}
       &&  \sum_n f^4(|A|)x_n,x_n\rangle^{r/2} \langle g^4(|A^*|)x_n,x_n\rangle^{r/2}\\
         &\leq & \left(\sum_n f^4(|A|)x_n,x_n\rangle^{p/2} \right)^{r/p} \left( \sum_n \langle g^4(|A^*|)x_n,x_n\rangle^{q/2}\right)^{r/q}\\
         &\leq&  \left(\sum_n f^{2p}(|A|)x_n,x_n\rangle^{} \right)^{r/p} \left( \sum_n \langle g^{2q}(|A^*|)x_n,x_n\rangle^{}\right)^{r/q} \,\, (\textit{by Lemma \ref{lemcarthy}})\\
         &\leq& \left( \textit{trace}\, f^{2p}(|A|) \right)^{r/p} \left( \textit{trace}\, g^{2q}(|A^*|) \right)^{r/q} \\
         &=& \| f(|A|)\|_{2p}^{2r} \| g(|A^*|)\|_{2q}^{2r}
        \end{eqnarray*}
        and 
      \begin{eqnarray*}
          \sum_n|\langle f^2(|A|)g^2(|A^*|)x_n,x_n \rangle|^r &\leq& \|f^2(|A|)g^2(|A^*|)\|_r^r.
      \end{eqnarray*}
Therefore,
\begin{eqnarray*}
         \sum_n|\langle Ax_n,x_n\rangle|^{2r} 
         &\leq& 
           \frac12 \left( \| f(|A|)\|_{2p}^{2r} \| g(|A^*|)\|_{2q}^{2r}+  \|f^2(|A|)g^2(|A^*|)\|_r^r \right).
     \end{eqnarray*}
This is the first inequality. The second inequality follows from the first inequality by considering $p=q$ (and then $r=p/2$).
      
\end{proof}

If we consider $f(t)=t^{\alpha}$ and $g(t)=t^{1-\alpha}$, $\alpha \in (0,1)$ in Theorem \ref{thp100}, then we get the following results.

\begin{cor}\label{corppp}
    Let $A\in \mathcal{K}(\mathcal{H})$. For any orthonormal set $\{x_n\} \subset \mathcal{H},$ 
\begin{eqnarray*}
        \sum_{n}|\langle Ax_n,x_n\rangle|^{2r} &\leq& \frac12 \| A\|_{2\alpha p}^{2r \alpha } \, \| A\|_{2(1-\alpha)q}^{2r (1-\alpha)} + \frac12 \| |A|^{2\alpha} |A^*|^{2(1-\alpha)}\|_{r}^{r},
    \end{eqnarray*}

where $\frac1r=\frac1p+\frac1q$ and $r\geq 1, p\geq 2, q\geq 2$ and $\alpha \in (0,1)$.
For $p=q,$ 
    \begin{eqnarray*}
        \sum_{n}|\langle Ax_n,x_n\rangle|^{p} &\leq& \frac12 \| A\|_{2\alpha p}^{\alpha p} \, \| A\|_{2(1-\alpha)p}^{(1-\alpha)p} + \frac12 \| |A|^{2\alpha} |A^*|^{2(1-\alpha)}\|_{p/2}^{p/2}.
    \end{eqnarray*} 


    \end{cor}

In particular, for $\alpha=\frac12$ in Corollary \ref{corppp}, we get the following results.

\begin{cor}
     Let $A\in \mathcal{K}(\mathcal{H})$. For any orthonormal set $\{x_n\} \subset \mathcal{H},$ 
     \begin{eqnarray*}
        \sum_{n}|\langle Ax_n,x_n\rangle|^{2r} &\leq& \frac12 \| A\|_{ p}^{ r} \| A\|_{q}^{r} + \frac12 \| A^2\|_{r}^{r},
    \end{eqnarray*} 

    where $\frac1r=\frac1p+\frac1q$ and $r\geq 1, p\geq 2, q\geq 2.$
For $p=q,$
    
    \begin{eqnarray}\label{p-norm}
        \sum_{n}|\langle Ax_n,x_n\rangle|^{p} &\leq& \frac12 \| A\|_{ p}^{ p}  + \frac12 \| A^2\|_{p/2}^{p/2}.
    \end{eqnarray} 
\end{cor}

\begin{remark}
    For any orthonormal set $\{x_n\} \subset \mathcal{H},$  from the inequality \eqref{p-norm}, we get 
    \begin{eqnarray}\label{p-norm11}
        \left( \sum_{n}|\langle Ax_n,x_n\rangle|^{p}\right)^{1/p} &\leq& \frac1{2^{1/p}} \left( \| A\|_{ p}^{ }  + \sqrt{ \| A^2\|_{p/2}^{}}\right),\quad \textit{ for all $p\geq 2.$}
    \end{eqnarray}

\end{remark}





From the inequality \eqref{p-norm}, we obtain the following result in terms of the eigenvalues.


\begin{cor}\label{cor..}
    Let $A\in \mathcal{K}(\mathcal{H}).$ If $\{ \lambda_n(A) \}$ is a listing of all non-zero eigenvalues (with multiplicity) of $A$, then
 \begin{eqnarray}\label{p-norm01}
        \sum_{n} \left|\lambda_n(A)\right|^{p} &\leq& \frac12 \| A\|_{ p}^{ p}  + \frac12 \| A^2\|_{p/2}^{p/2},\quad \textit{ for all $p\geq 2.$}
    \end{eqnarray}


\end{cor}

\begin{proof}
Following  \cite[equation (5)]{B.Simon}, there exists an orthonormal set (known as Schur “basis”) $\{x_n\}\subset \mathcal{H}$ such that $ \lambda_n(A)= \langle Ax_n,x_n\rangle.$
Therefore, the desired inequality \eqref{p-norm01} follows from \eqref{p-norm}.
\end{proof}

We now prove the following proposition.

\begin{prop}\label{0prop}
    If $A\in \mathcal{K}(\mathcal{H}),$ then
    \begin{eqnarray}\label{0p0p0}
         \| A^2\|_{p/2}^{p/2} \leq \|A\|_p^p, \quad \textit{for all $p\geq 2.$}
    \end{eqnarray}
    
\end{prop}

\begin{proof}
    Suppose $\{x_n\}$ and $\{y_n\}$ are any two orthonormal sets in $\mathcal{H}.$
    Then, we have
    \begin{eqnarray*}
        \|A^2\|_{p/2}^{p/2} &=& \sup \sum_n |\langle A^2x_n,y_n\rangle|^{p/2}\\
        &&\,\, (\textit{supremum is taken over all orthonormal sets $\{x_n\}$ and $\{y_n\}$ })\\
        &\leq&  \sup \sum_n \| Ax_n\|^{p/2}\|A^*y_n\|^{p/2} \,\,(\textit{by Cauchy-Schwarz inequality})\\
        &=& \sup \sum_n \langle |A|^2x_n,x_n\rangle|^{p/4} \langle |A^*|^2y_n,y_n\rangle|^{p/4}\\
        &\leq& \sup \sum_n \langle |A|^px_n,x_n\rangle|^{1/2} \langle |A^*|^py_n,y_n\rangle|^{1/2} \,\, (\textit{by Lemma \ref{lemcarthy}})\\
        &\leq& \sup \sum_n \frac12( \langle |A|^px_n,x_n\rangle+ \langle |A^*|^py_n,y_n\rangle ) \,\,(\textit{by AM-GM inequality})\\
        &\leq&  \frac12\left( \sup \sum_n \langle |A|^px_n,x_n\rangle|+ \sup \sum_n \langle |A^*|^py_n,y_n\rangle| \right)\\
        &\leq& \frac12 (\textit{trace}\, |A|^p+ \textit{trace}\, |A^*|^p)  \\
        &=& \|A\|_p^p,
    \end{eqnarray*}
as desired.
\end{proof}

\begin{remark}
  Weyl's inequality \cite{Weyl} (see  \cite[Th. 2.3]{B.Simon} for another proof) says that 
   \begin{eqnarray}\label{improve020}
        \sum_{n}\left|\lambda_n(A)\right|^{p} &\leq&  \| A\|_{ p}^{ p}, \quad \textit{for all $p\geq 1$.}
    \end{eqnarray}
    Following  the inequality \eqref{0p0p0}, we have 
    \begin{eqnarray*}
       \frac12 \| A\|_{ p}^{ p}  + \frac12 \| A^2\|_{p/2}^{p/2} &\leq &\|A\|_p^p,  \quad \textit{for all $p\geq 2$}.
    \end{eqnarray*}
    Therefore, the inequality \eqref{p-norm01} refines  Weyl's inequality \eqref{improve020} for all $p\geq 2$.

\end{remark}

 In the following, we obtain another generalization of Weyl's inequality \eqref{improve020}.

\begin{theorem}\label{th9999}
    Let $A\in \mathcal{K}(\mathcal{H})$ and let $f$, $g$ be as in Lemma \ref{lemkitt}. 
    If $\{ \lambda_n(A) \}$ is a listing of all non-zero eigenvalues (with multiplicity) of $A$, then
 \begin{eqnarray}\label{p-norm019}
        \sum_{n} \left|\lambda_n(A)\right|^{p} &\leq& \frac12 \| f^{2}(|A|) \|_p^p + \frac12 \|g^{2}(|A^*|)\|_p^p, \quad \textit{ for all $p\geq 1.$ }
    \end{eqnarray}
Note that Weyl's inequality is reduced in the special case  $f(t)=g(t)=t^{1/2}.$
    
\end{theorem}

\begin{proof}
Let $\{x_n\} \subset \mathcal{H}$ be any orthonormal set. Then from Lemma \ref{lemkitt}, we have
\begin{eqnarray*}
    \sum_n|\langle Ax_n,x_n\rangle|^p &\leq& \sum_n\langle f^2(|A|)x_n,x_n\rangle^{p/2}   \langle g^2(|A^*|)x_n,x_n\rangle^{p/2}\\  
    &\leq& \frac12 \sum_n (\langle f^2(|A|)x_n,x_n\rangle^{p} +  \langle g^2(|A^*|)x_n,x_n\rangle^{p})\\
    &\leq& \frac12 \sum_n (\langle f^{2p}(|A|)x_n,x_n\rangle^{} +  \langle g^{2p}(|A^*|)x_n,x_n\rangle^{})\\
    &\leq& \frac12 \textit{trace}\,(f^{2p}(|A|)+g^{2p}(|A^*|))\\
    &=& \frac12 \| f^{2p}(|A|)+g^{2p}(|A^*|)\|_1\\
    &=& \frac 12 \| f^{2}(|A|) \|_p^p + \frac 12\|g^{2}(|A^*|)\|_p^p.
\end{eqnarray*}
Therefore, the inequality 
$\sum_n|\langle Ax_n,x_n\rangle|^p
    \leq \frac12 \| f^{2}(|A|) \|_p^p + \frac12 \|g^{2}(|A^*|)\|_p^p$
is true for every orthonormal set $\{x_n\} \subset \mathcal{H}.$ The desired inequality \eqref{p-norm019} follows by taking the  Schur ``basis''   $\{x_n\} \subset \mathcal{H}$.  
\end{proof}


For $A\in \mathcal{K}(\mathcal{H}),$  Simon \cite[Th. 2.3]{B.Simon} provided the following inequality: for any two orthonormal sets $\{x_n\}$ and $\{y_n\}$ in $\mathcal{H}$,
$$\sum_n|\langle Ax_n,y_n\rangle|^p\leq \|A\|_p^p, \quad \textit{for all} \,\, p\geq 1.$$

We now provide a similar type of inequality for the finite rank operators, which follows from the proof of Theorem \ref{th100p}.

\begin{theorem}\label{thhhhhp}
    Let $A\in \mathcal{B}(\mathcal{H})$ with finite rank. Then for any two orthonormal sets $\{x_n\}$ and $\{y_n\}$ in $\mathcal{H}$, we have
    \begin{eqnarray}
        \sum_n|\langle Ax_n,y_n\rangle|^{2r} &\leq& (\textit{rank}\, A )^{r/q}\|A\|_{2p}^{2r},
    \end{eqnarray}
        
where $1\leq p,q<\infty$ and $\frac1r=\frac1p+\frac1q.$ 
In particular, for $p=q,$

\begin{eqnarray}\label{pi0}
    \sum_n|\langle Ax_n,y_n\rangle|^{p} &\leq& (\textit{rank}\, A )^{1/2}\|A\|_{2p}^{p},\quad \textit{for all $p\geq 1.$
}
\end{eqnarray}
\end{theorem}

By considering $x_n=y_n$ in $\eqref{pi0}$ and  $\{x_n\}\subset \mathcal{H}$ as a Schur ``basis'', we get:

\begin{cor}\label{cor99990}
     Let $A\in \mathcal{B}(\mathcal{H})$ with finite rank. If $\{ \lambda_n(A) \}$ is a listing of all non-zero eigenvalues (with multiplicity) of $A$, then
 \begin{eqnarray}\label{p-norm01000}
        \sum_{n} \left|\lambda_n(A)\right|^{p} &\leq& (\textit{rank}\, A )^{1/2}\|A\|_{2p}^{p}, \quad \textit{for all $p\geq 1.$}
\end{eqnarray}

\end{cor}


For $A\in \mathcal{K}(\mathcal{H})$, it is well known that the function $p\to \|A\|_p$ is monotone decreasing on $p\geq 1$ (i.e., $\|A\|_q\leq \|A\|_p$ for $1\leq p <q)$.
We now bound $\|A\|_p$ from above for $A\in \mathcal{M}_n(\mathbb{C})$, from which we obtain a monotone increasing function involving the Schatten $p$-norm. 

 \begin{theorem}\label{thmatrix}
   Let $1\leq p<q <\infty.$  If $A\in \mathcal{M}_n(\mathbb{C}),$ then $\|A\|_p \leq n^{\frac{q-p}{pq}} \|A\|_q .$ In particular, the function $p\to n^{-{1}/{p}}\|A\|_p$ is non-decreasing in $p\in [1,\infty).$

\end{theorem}
\begin{proof}
    Let $p'=\frac qp$ and $q'=\frac{q}{q-p}.$ Then clearly $p'>1$ and $q'>0$ with $\frac{1}{p'}+\frac{1}{q'}=1.$  Let $\{x_1,x_2,\ldots,x_n\}$ and $\{y_1,y_2,\ldots,y_n\}$ be any two orthonormal sets in $\mathbb{C}^n.$ Then
    from H\"older's inequality, we obtain
\begin{eqnarray*}
         \sum_{k=1}^n |\langle Ax_n,y_n\rangle|^{p} &\leq& \left( \sum_{k=1}^n |\langle Ax_n,y_n\rangle|^{pp'} \right)^{1/{p'}} n^{1/{q'}}
         =\left( \sum_{k=1}^n |\langle Ax_n,y_n\rangle|^{q} \right)^{p/{q}} n^{\frac{q-p}{q}},
\end{eqnarray*}
which implies the desired inequality. 
\end{proof}


\section{Operator norm and numerical radius inequalities}\label{sec5}

In this section, we obtain the operator norm and numerical radius inequalities of bounded linear operators, which generalize and improve the existing inequalities \eqref{2002jot}, \eqref{k03} and \eqref{pintu22}. In order to prove our results first we need the following lemma.

\begin{lemma} \cite{Hou1995} \label{lem-Hou}
	Let $A,B,X,Y\in \mathcal{B}(\mathcal{H}).$ Then
	$$r\left( \begin{bmatrix}
	A&X\\
	B&Y
	\end{bmatrix}\right) \leq r\left( \begin{bmatrix}
	\|A\| &\|X\| \\
	\| B\|& \|Y\|
	\end{bmatrix}\right) .$$
\end{lemma}

We can now obtain a generalization of the inequality \eqref{2002jot}.

\begin{theorem}\label{th5}
	Let $X,Y\in \mathcal{B}(\mathcal{H})$ be positive. Then
	\begin{eqnarray}
	\|X+Y\| \leq \max(\|X\|,\|Y\| )+ w\left( \begin{bmatrix}
	0& X^{1-t}Y^{1-\alpha}\\
	Y^{\alpha}X^t & 0
	\end{bmatrix}  \right)
	\end{eqnarray}
	and
	\begin{eqnarray}\label{r11}
	\|X+Y\| \leq \frac{ \|X\|+\|Y\| + \sqrt{ (\|X\|-\|Y\|)^2 +4 \| X^{1-t}Y^{1-\alpha} \| \| X^{t}Y^{\alpha}\| } }{2},
	\end{eqnarray}
	for all $\alpha, t\in [0,1].$ In particular, for $\alpha=t=\frac12,$
	\begin{eqnarray*}\label{}
	\|X+Y\| \leq \frac{ \|X\|+\|Y\| + \sqrt{ (\|X\|-\|Y\|)^2 +4 \| X^{1/2}Y^{1/2} \|^2  } }{2}.
	\end{eqnarray*}
\end{theorem}

\begin{proof}
	We have 
	\begin{eqnarray}\label{spec}
	  \|X+Y\| \notag 
	&=& \left \|\begin{bmatrix}
	X+Y &0\\
	0&0
	\end{bmatrix}  \right\| \notag\\
	&=& r \left (\begin{bmatrix}
	X+Y &0\\
	0&0
	\end{bmatrix}  \right)\notag\\
	&=& r \left (\begin{bmatrix}
	X^t &Y^{1-\alpha}\\
	0&0
	\end{bmatrix} \begin{bmatrix}
	X^{1-t}&0\\
	Y^{\alpha}&0
	\end{bmatrix} \right)\notag\\
	&=& r \left (\begin{bmatrix}
	X^{1-t}&0\\
	Y^{\alpha}&0
	\end{bmatrix} \begin{bmatrix}
	X^t &Y^{1-\alpha}\\
	0&0
	\end{bmatrix}  \right) \,\, (\text{$r(AB)=r(BA)$  $\forall A,B\in \mathcal{B}(\mathcal{H})$)} \notag\\
	&=& r\left(\begin{bmatrix}
	X& X^{1-t}Y^{1-\alpha}\\
	Y^{\alpha}X^t & Y
	\end{bmatrix}\right)\\
	&=& w\left(\begin{bmatrix}
	X& X^{1-t}Y^{1-\alpha}\\
	Y^{\alpha}X^t & Y
	\end{bmatrix}\right) \notag\\
	&\leq& w\left(\begin{bmatrix}
	X& 0\\
	0 & Y
	\end{bmatrix}\right) + w\left(\begin{bmatrix}
	0& X^{1-t}Y^{1-\alpha}\\
	Y^{\alpha}X^t & 0
	\end{bmatrix}\right) \notag\\
	&=& \max (\|X\|, \|Y\|)+ w\left(\begin{bmatrix}
	0& X^{1-t}Y^{1-\alpha}\\
	Y^{\alpha}X^t & 0
	\end{bmatrix}\right) \notag.
	\end{eqnarray}
	This is the first inequality. Again, from \eqref{spec}, we have
	\begin{eqnarray*}
		\|X+Y\| &\leq& r\left(\begin{bmatrix}
			\|X\| & \|X^{1-t}Y^{1-\alpha} \| \\
			\|Y^{\alpha}X^t\| & \|Y\|
		\end{bmatrix}\right) \,\, (\text{using Lemma \ref{lem-Hou}})\\
		&=& \frac{ \|X\|+\|Y\| + \sqrt{ (\|X\|-\|Y\|)^2 +4 \| X^{1-t}Y^{1-\alpha} \| \| X^tY^{\alpha}\| } }{2},
	\end{eqnarray*}
	This completes the proof.
\end{proof}

Clearly, \eqref{r11} generalizes the existing inequality \eqref{2002jot}.
Now using \eqref{r11} we develop an upper bound for the numerical radius in terms of the operator norm and the spectral radius. To prove this we need the mixed Schwarz inequality (see Lemma \ref{lemkitt})
\begin{eqnarray}\label{mixed}
| \langle Ax,x\rangle|^2 \leq \langle |A|x,x\rangle \langle |A^*|x,x\rangle , \, \text{ where $x\in \mathcal{H}$ with $\|x\|=1.$}
\end{eqnarray}

\begin{theorem}\label{corp11}
	Let $A\in \mathcal{B}(\mathcal{H}).$ Then
	\begin{eqnarray*}
		w(A) \leq \frac12 \|A\|+ \frac12 r^{1/4} \left( |A|^{2t}|A^*|^{2\alpha}\right ) \times r^{1/4}\left ( |A|^{2(1-t)}|A^*|^{2(1-\alpha)}\right),
	\end{eqnarray*}
 
	for all $\alpha,t\in [0,1].$ In particular, for $\alpha=t=\frac12,$
	\begin{eqnarray}\label{pintu}
	w(A) \leq \frac12 \|A\|+ \frac12 {r^{1/2} \left( |A|^{}|A^*|^{}\right )}.
	\end{eqnarray}
	
Also, in particular, for $\alpha=t=0,$
	\begin{eqnarray}\label{kit3}
	w(A) \leq \frac12 \|A\|+ \frac12 {\left\|A^2\right\|^{1/2}}.
	\end{eqnarray}
\end{theorem}
\begin{proof}
	Take $x\in \mathcal{H}$ with $\|x\|=1.$ From \eqref{mixed} and using the AM-GM inequality, we get
	$$|\langle Ax,x\rangle| \leq \frac12 \langle (|A|+|A^*|x,x\rangle) \leq \frac12 \|  |A|+ |A^*| \| .$$
	Therefore, taking supremum over $\|x\|=1,$ we get
	\begin{eqnarray}\label{1}
	w(A) \leq \frac12 \|  |A|+ |A^*| \|.
	\end{eqnarray}
	Considering $X=|A|$ and $Y=|A^*|$ in  \eqref{r11}, we obtain 
	\begin{eqnarray*}
		\||A|+|A^*|\| \leq { \|A\|+ { \| |A|^{1-t}|A^*|^{1-\alpha} \|^{1/2} \| |A|^t|A^*|^{\alpha}\|^{1/2} }  }, \quad \forall \alpha, t\in [0,1].
	\end{eqnarray*}
	Now, $\||A|^{t}|A^*|^{\alpha}\|^2 =  \| |A^*|^{\alpha} |A|^{t}  |A|^{t} |A^*|^{\alpha}  \|
	= r(|A^*|^{\alpha} |A|^{t}  |A|^{t} |A^*|^{\alpha})
	= r( |A|^{2t} |A^*|^{2\alpha} ).$
	Similarly, $
		\| |A|^{1-t}|A^*|^{1-\alpha} \|^2 = r( |A|^{2(1-t)}|A^*|^{2(1-\alpha)} ).
	$
	Hence, for all $\alpha, t\in [0,1],$ we get
	\begin{eqnarray}\label{2}
	\||A|+|A^*|\| \leq { \|A\|+ { r^{1/4}( |A|^{2(1-t)}|A^*|^{2(1-\alpha)} ) \times r^{1/4}( |A|^{2t} |A^*|^{2\alpha} ) } }.
	\end{eqnarray}
	Combining the inequalities \eqref{1} and \eqref{2}, we get
	\begin{eqnarray*}
		w(A) \leq \frac12 \|A\|+ \frac12 {r^{1/4} \left( |A|^{2t}|A^*|^{2\alpha}\right ) \times r^{1/4}\left ( |A|^{2(1-t)}|A^*|^{2(1-\alpha)}\right)},
	\end{eqnarray*}
	for all $\alpha,t\in [0,1].$ The second inequality follows by considering $\alpha=t=\frac12$ and the third inequality follows by considering $\alpha=t=0.$
\end{proof}

\begin{remark}
(i)	It is easy to verify that 
	\begin{eqnarray*}
		\frac12 \|A\|+ \frac12 r^{1/4} \left( |A|^{2t}|A^*|^{2\alpha}\right ) \times r^{1/4}\left ( |A|^{2(1-t)}|A^*|^{2(1-\alpha)}\right) &\leq& \|A\|,
	\end{eqnarray*}
	for all $\alpha,t\in [0,1].$
	Therefore, the numerical radius bound in Theorem \ref{corp11} improves the bound $w(A)\leq \|A\|.$ 
	

(iii) Considering $\alpha =t$ in Theorem \ref{corp11}, we see that
\begin{eqnarray*}
		w(A) &\leq& \frac12 \|A\|+ \frac12 {r^{1/4} \left( |A|^{2t}|A^*|^{2t}\right ) \times r^{1/4}\left ( |A|^{2(1-t)}|A^*|^{2(1-t)}\right)}\\
		&=& \frac12 \|A\|+ \frac12  \left\| |A|^{t}|A^*|^{t}\right \|^{1/2}  \left\| |A|^{1-t}|A^*|^{1-t}\right \|^{1/2}\\
	&\leq& 	 \frac12 \|A\|+ \frac12  \left\| |A|^{}|A^*|^{}\right \|^{t/2}  \left\| |A|^{}|A^*|^{}\right \|^{(1-t)/2} \,\, (\textit{by Heinz inequality})\\
	&=& \frac12 \|A\|+ \frac12 \left\| A^2\right\|^{1/2}, \, \text{for all $t\in [0,1].$}
\end{eqnarray*}
	Therefore, Theorem \ref{corp11} improves as well as generalizes the inequality \eqref{k03}. 

\end{remark}

We next obtain improvements of the triangle inequality for the operator norm.

\begin{theorem}\label{th4}
	Let $A,B\in \mathcal{B}(\mathcal{H})$. Then
	\begin{eqnarray*}
		\|A\pm B\| \leq \sqrt{\| A^*A+B^*B\|+ 2w(A^*B)}
	\end{eqnarray*} 
	and 
	\begin{eqnarray*}
		\|A\pm B\| \leq \sqrt{\| AA^*+BB^*\|+ 2w(AB^*)}
	\end{eqnarray*}
\end{theorem}
\begin{proof}
	Let $x\in \mathcal{H}$ and $\|x\|=1.$ Then
	\begin{eqnarray*}
		\|(A+B)x\|^2&=& \langle Ax+Bx, Ax+Bx\rangle\\
		&=& \|Ax\|^2+\|Bx\|^2+ \langle Ax, Bx\rangle + \langle Bx,Ax\rangle\\
		&=& \langle (A^*A+B^*B)x,x\rangle+ 2 \langle \Re(A^*B)x,x\rangle\\
		&=& \langle (A^*A+B^*B)x,x\rangle+ 2 |\langle \Re(A^*B)x,x\rangle|\\
		&\leq & \|A^*A+B^*B\| + 2\|\Re(A^*B)\|\\
		&\leq & \|A^*A+B^*B\| + 2\sup_{\theta \in \mathbb{R}}\|\Re(e^{i\theta }A^*B)\|\\
		&=& \|A^*A+B^*B\| + 2 w(A^*B).
	\end{eqnarray*}
	Therefore, taking the supremum over $\|x\|=1,$ we get
	$$\|A+B\|^2\leq \|A^*A+B^*B\| + 2 w(A^*B),$$
	as desired. The second inequality follows by replacing $A$ by $A^*$ and $B$ by $B^*$.
	\end{proof}

\begin{remark}\label{rem01}
(i)	Clearly, we see that 
	$$\sqrt{\| A^*A+B^*B\|+ 2w(A^*B)}\leq \sqrt{\|A\|^2+\|B\|^2+ 2\|A^*B\|}\leq \|A\|+\|B\|$$
	
	and
	$$\sqrt{\| AA^*+BB^*\|+ 2w(AB^*)}\leq \sqrt{\|A\|^2+\|B\|^2+ 2\|AB^*\|}\leq \|A\|+\|B\|.$$ 
 
(ii)	The inequalities in Theorem \ref{th4} refine the following existing inequality
	$$\|A+B\| \leq \sqrt{\|A\|^2+ \|B\|^2+\|A\|\|B\| + \min \left( w(A^*B), w(AB^*)\right)},$$
	which is recently shown in \cite{Bhunia_Rocky}.
\end{remark}

From Theorem \ref{th4} we obtain a sufficient condition for the norm equality.

\begin{cor}
	Let $A\in \mathcal{B}(\mathcal{H}).$ If $\Re(A) \Im(A) =0,$ then
	$$\|A\|= \sqrt{\frac12 \|A^*A+AA^*\|}.$$
\end{cor}
\begin{proof}
	Substituting $A$ by $\Re(A)$ and $B$ by $i \Im(A)$ in Theorem \ref{th4}, we get
	$$\|A\|^2\leq \frac12 \|A^*A+AA^*\|.$$
	Also, $\frac12 \|A^*A+AA^*\|\leq \|A\|^2.$ So, $\|A\|^2= \frac12 \|A^*A+AA^*\|.$
\end{proof}

The improvements of the triangle inequality in Theorem \ref{th4} yield a necessary and sufficient condition for the parallelism of two bounded linear operators:

\begin{cor}\label{parallel}
	Let $A,B\in \mathcal{B}(\mathcal{H})$. Then
	$ A \parallel B \,\, \text{if and only if}\,\, w(A^*B)=\|A\|\|B\|.$
\end{cor}

We omit the proof as the result also follows from  \cite[Th. 3.3]{Zamani}.
Also since $A \parallel B$ if and only if $A^* \parallel B^*,$ using Corollary \ref{parallel} we also get the following result.
\begin{cor}\label{parallel2}
	Let $A,B\in \mathcal{B}(\mathcal{H})$. Then
	$ A \parallel B \,\, \text{if and only if}\,\, w(AB^*)=\|A\|\|B\|.$
\end{cor}

Again, using Theorem \ref{th4}, we deduce the following necessary conditions for the parallelism of two bounded linear operators.

\begin{cor}\label{parallel3}
	Let $A,B\in \mathcal{B}(\mathcal{H})$. If
	$A \parallel B$, then
	  
	$$ \|A^*A+B^*B\|= \|AA^*+BB^*\|= \|A\|^2+ \|B\|^2.$$

\end{cor}

\begin{proof}
	Let $ A \parallel B.$ Then there exists a scalar $\lambda$, $|\lambda|=1$ such that $\|A+\lambda B\|=\|A\|+\|B\|.$ Now 
	replacing $B$ by $\lambda B$ in Theorem \ref{th4}, we get 
	\begin{eqnarray*}
		\|A\|+\|B\|= \|A+\lambda B\| &\leq & \sqrt{\| A^*A+B^*B\|+ 2w(A^*B)}\\
		&\leq & \sqrt{\|A\|^2+\|B\|^2+ 2\|A^*B\|} \\
		&\leq& \sqrt{\|A\|^2+\|B\|^2+ 2\|A\| \|B\|}\\
		&=& \|A\|+\|B\|.
	\end{eqnarray*}
	This implies that $ \|A^*A+B^*B\|=  \|A\|^2+ \|B\|^2.$ Again, since $ A \parallel B$ if anly if  $ A^* \parallel B^*,$ we get $\|AA^*+BB^*\|= \|A\|^2+ \|B\|^2.$

\end{proof}

However, the converse part is not true, in general. For example, let $A=(a_{ij})$ be an $n\times n$ matrix,  where $a_{ij}=\delta_{j,i+1}$  and let $B=A^*.$ Then $\|A^*A+B^*B\|= \|AA^*+BB^*\|= \|A\|^2+ \|B\|^2=2$, but 
$$\max_{|\lambda|=1}\|A+\lambda B\|= \max_{|\lambda|=1}\|A+\lambda A^*\|=2w(A)=2 \cos\left(\frac{\pi}{n+1}\right)< 2= \|A\|+\|B\|.$$


\section{A bound for sum of the zeros of a polynomial}\label{sec6}

\noindent
In this section, as an application of the Schatten $p$-norm inequalities, we give an upper bound for the sum of the absolute values of the zeros of a complex polynomial
  $$ p(z) = z^n + a_nz^{n-1} + \ldots + a_2z + a_1 $$ of degree $n\geq 2 $ with $a_1 \neq 0$.                                     Various bounds for the zeros of $p(z)$ have been studied by many mathematicians over the years by using  the numerical radius inequalities to the Frobenius companion matrix associated with the polynomial $p(z)$, see \cite{Book2022}. Recall that the Frobenius companion matrix of the polynomial $p(z)$ is  $$C(p)=\begin{pmatrix}
-a_n & -a_{n-1} & .... & -a_2 & -a_1\\
1 & 0 & ...& 0 & 0\\
0 & 1 & ... & 0 & 0\\
\vdots & \vdots & \ddots & \vdots & \vdots\\
0 & 0 & .... & 1 & 0
\end{pmatrix}.$$
The characteristic polynomial of $ C(p)$ is the polynomial $p(z)$ and so the zeros of $p(z)$ are exactly the eigenvalues of $C(p)$, see \cite[p. 316]{a14}. Now we are in a position to prove: 

\begin{theorem}\label{poly}
    If $\lambda_1,\lambda_2,\ldots, \lambda_n$ are the zeros of  $p(z),$ then
    \begin{eqnarray*}
      |\lambda_1|+|\lambda_2|+\ldots+|\lambda_n| &\leq& \sqrt{n\left(n-1 + \sum_{i=1}^n|a_i|^2 \right)}.
    \end{eqnarray*}
\end{theorem}

\begin{proof}
   Following the Schatten $p$-norm inequality \eqref{p-norm01000} for $p=1$, we get
$$ \sum_{j=1}^n|\lambda_j| \leq  \left(\textit{rank} \,\, C(p) \right)^{1/2}\|C(p)\|_{2}^{}
        =   \sqrt{n\left(\textit{trace} \,\, |C(p)|^2\right)}
        = \sqrt{n\left(n-1 + \sum_{i=1}^n|a_i|^2 \right)},$$
as desired.
\end{proof}

As a consequence of Theorem \ref{poly}, we get an upper bound for the smallest absolute value of the zeros of the polynomial $p(z).$

\begin{cor}
 If $\lambda_1,\lambda_2,\ldots, \lambda_n$ are the zeros of  $p(z)$ with $|\lambda_1|\leq |\lambda_2| \leq \ldots\leq | \lambda_n|$, then
    \begin{eqnarray*}
      |\lambda_1| &\leq& \frac1n \sqrt{n\left(n-1 + \sum_{i=1}^n|a_i|^2 \right)}.
    \end{eqnarray*}   
\end{cor}

\section{A bound for the energy of a graph}\label{sec7}

\noindent 
In the final section, as an application of the Schatten $p$-norm inequalities, we obtain a refined upper bound for the energy of a simple graph.  Let $G$ be a simple undirected graph with vertex set $V(G)=\{v_1,v_2,\ldots,v_n\}$ and edge set $E(G)=\{e_1,e_2,\ldots, e_m\}.$ Let $d_i$ be the degree of the vertex $v_i,$ for $i=1,2,\ldots,n.$
The adjacency matrix associated with the graph $G$, denoted as $\textit{Adj}(G)$, is defined as $\textit{Adj}(G)=(a_{ij})_{n\times n}$, where $a_{ij} = 1,$ if $ v_i \sim v_j$ (i.e., $v_i$ is adjacent to $v_j$)  and $a_{ij} = 0$ otherwise. Clearly, $\textit{Adj}(G)$ is a symmetric  matrix with entries $0,1$ and the main diagonal entries are zero. Let $\lambda_1, \lambda_2,\ldots, \lambda_n$ be the eigenvalues of $\textit{Adj}(G).$
The energy of the graph $G$, denoted as $\mathcal{E}(G)$, is defined as $$\mathcal{E}(G)=\sum_{i=1}^n|\lambda_i|.$$
This concept was introduced by Gutman \cite{Gutman}, in connection to the
total $\pi$-electron energy. For details on the general theory of the total $\pi$-electron energy, as well as its chemical applications, see \cite{Appl1, Appl2}. The search of upper and lower bounds for $\mathcal{E}(G)$ is a wide subfield of spectral graph theory.
In \cite{Energy},  McClelland showed that 
\begin{eqnarray}\label{eng}
    \mathcal{E}(G) &\leq& \sqrt{2mn}.
\end{eqnarray}
After that, various bounds have been studied in the literature, we refer to see \cite{Jahan, Rada1} and the references therein.
Here, we provide a new upper bound of $\mathcal{E}(G)$ by using  the rank of the matrix $\textit{Adj}(G)$), which is strictly stronger than the bound \eqref{eng} when the graph is singular.

\begin{theorem}\label{graph1}
   Let $G$ be a simple graph. Then 
\begin{eqnarray*}
     \mathcal{E}(G) &\leq&  \sqrt{2m \left(\textit{rank Adj}(G) \right)}.
\end{eqnarray*}

\end{theorem}

\begin{proof}
    From the Schatten $p$-norm inequality \eqref{p-norm01000} (for the case $p=1$), we get
    \begin{eqnarray*}
        \sum_{i=1}^n|\lambda_i| &\leq& \left(\textit{rank}\, \textit{Adj}(G)\right)^{1/2}  \|A(G)\|_2\\
       & =& \left(\textit{rank}\, \textit{Adj}(G)\right)^{1/2} \,\, \left(\textit{trace}\, |\textit{Adj}(G)|^2 \right)^{1/2}\\
        &=& \sqrt{(\textit{rank Adj}(G)) \sum_{i=1}^n d_i} \quad \left( \textit{since} \,\, \textit{trace}\, |\textit{Adj}(G)|^2= \sum_{i=1}^n d_i\right)\\
        &=& \sqrt{2m (\textit{rank Adj}(G))} \quad \left( \textit{since} \,\, \sum_{i=1}^n d_i= 2m \right),
        \end{eqnarray*}
as desired.
\end{proof}

\vspace{1cm}



\noindent \textbf{Declaration of competing interest.}
The author declares that there is no competing interest.



\bibliographystyle{amsplain}

\begin{thebibliography}{99}


\bibitem{Natoor2022} A. Al-Natoor, S. Benzamia, and F. Kittaneh, Unitarily invariant norm inequalities for positive semidefnite matrices, Linear Algebra Appl. 633 (2022), 303--315.

\bibitem{kit1} A. Al-Natoor and F. Kittaneh, Further unitarily invariant norm inequalities for positive semidefinite matrices, Positivity 26 (2022), Paper No. 8, 11 pp.
	
	

\bibitem{Alomari} M. W. Alomari, S. Sahoo, and M. Bakherad, Further numerical radius inequalities, J. Math. Inequal. 16 (2022), no. 1, 307–326.
	


\bibitem{Penrose1} A. Ben-Israel, The Moore of the Moore-Penrose inverse, Electron. J. Linear Algebra., 9 (2002), 150--157.
	
	\bibitem{Bhatia1990} R. Bhatia, and F. Kittaneh, On the singular values of a product of operators, SIAM J. Matrix Anal. Appl. 11 (1990) 272--277.

 \bibitem{Bhunia2} P. Bhunia, Numerical radius and spectral radius inequalities with an estimation for roots of a polynomial, Indian J. Pure Appl. Math. (2023). https://doi.org/10.1007/s13226-023-00523-x
	
	\bibitem{Bhunia2024} P. Bhunia, Improved bounds for the numerical radius via polar decomposition of operators, Linear Algebra Appl. 683 (2024), 31--45. 
	
	\bibitem{Book2022} P. Bhunia, S. S. Dragomir, M. S. Moslehian and K. Paul, Lectures on numerical radius inequalities, Infosys Science Foundation Series in Mathematical Sciences, Springer, Cham, 2022.
	
	
	
	
 
	\bibitem{FAA} P. Bhunia, S. Jana, M. S. Moslehian and K. Paul, Improved Inequalities for Numerical Radius via Cartesian Decomposition, Funct. Anal. Appl. 57 (2023), no. 1, 18–28.
	
	\bibitem{withkittaneh} P. Bhunia, F. Kittaneh, K. Paul and A. Sen, Anderson's theorem and A-spectral radius bounds for semi-Hilbertian space operators, Linear Algebra Appl. 657 (2023), 147--162.
	
	
	
	
	
	
	
	\bibitem{Bhunia_Rocky} P. Bhunia and K. Paul, Refinements of norm and numerical radius inequalities, Rocky Mountain J. Math. 51 (2021), no. 6, 1953-1965.
	
 \bibitem{Bhunia2021} P. Bhunia and K. Paul, Furtherance of numerical radius inequalities of Hilbert space operators, Arch. Math. (Basel) 117 (2021), no. 5, 537–546.

 \bibitem{Buzano} M.L. Buzano, Generalizzatione della disuguaglianza di Cauchy-Schwarz, Rend. Sem. Mat. Univ. e Politech. Torino 31(1971/73) (1974) 405-409.
	
	\bibitem{Feki} K. Feki, and F. Kittaneh, Some new refinements of generalized numerical radius inequalities for Hilbert space operators, Mediterr. J. Math. 19 (2022), no. 1, Paper No. 17, 16 pp. 
	
	
	\bibitem{GUS1997} K. E. Gustafson, and D. K. M. Rao, Numerical range. The field of values of linear operators and matrices, Universitext, Springer-Verlag, New York, 1997.

 


 \bibitem{Gutman} I. Gutman, The energy of a graph, Ber. Math.-Statist. Sekt. Forshungszentrum Graz 103 (1978) 1–22.


\bibitem{Appl1} I. Gutman, Bounds for total $\pi$-electron energy of conjugated hydrocarbons, Z. Phys. Chem. (Leipzig) 226 (1985) 59–64.

\bibitem{Appl2} I. Gutman and O.E. Polansky, Mathematical Concepts in Organic Chemistry, Springer, Berlin, 1986.

\bibitem{lap} I. Gutman and B. Zhou,  Laplacian energy of a graph, Linear Algebra Appl. 414 (2006) 29--37. 
	
	

 \bibitem{a14} R.A. Horn and C.R. Johnson, Matrix Analysis, Cambridge, 1985.

 \bibitem{Hou1995} J. C. Hou and H. K. Du, Norm inequalities of positive operator matrices, Integral Equations Operator Theory 22 (1995) 281--294.

 \bibitem{Jahan} A. Jahanbani, Upper Bounds for the Energy of Graphs,  MATCH Commun. Math. Comput. Chem. 79 (2018) 275-286.
	
\bibitem{Kittaneh1988} F. Kittaneh, Notes on some inequalities for Hilbert space operators, Publ. Res. Inst. Math. Sci. 24 (1988), no. 2, 283--293.

 	\bibitem{Kittaneh1992} F. Kittaneh, A note on the arithmetic-geometric mean inequality for matrices, Linear
	Algebra Appl. 171 (1992), 1--8.
	
	\bibitem{kittaneh1997} F. Kittaneh, Norm inequalities for certain operator sums, J. Funct. Anal. 143 (1997), 337-348.
	
	\bibitem{kittaneh2002} F. Kittaneh, Commutator inequalities associated with the polar decomposition, Proc. Amer. Math.
	Soc. 130 (2002) 1279--1283.

 \bibitem{JOT2002} F. Kittaneh,  Norm inequalities for sums of positive operators, J. Operator Theory 48 (2002), no. 1, 95-103. 

 \bibitem{Kittaneh_2003} F. Kittaneh, Numerical radius inequality and an estimate for the numerical radius of the Frobenius companion matrix, Studia Math. 158 (2003), no. 1, 11--17.
	
	
	\bibitem{Kittaneh_LAMA_2023} F. Kittaneh, H.R. Moradi and M. Sababheh, Sharper bounds for the numerical radius, Linear Multilinear Algebra, (2023). https://doi.org/10.1080/03081087.2023.2177248

 \bibitem{Energy} B. J. McClelland, Properties of the latent roots of a matrix: The estimation of $\pi$ electron energies, J. Chem. Phys. 54 (1971) 640--643.
	
	\bibitem{Penrose2} R. Penrose, A generalized inverse for matrices, Math. Proc. Cambridge Philos. Soc., 51 (1955), 406--413.


  \bibitem{Rada1} J. Rada and A. Tineo, Upper and lower bounds for the energy of bipartite graphs, J. Math. Anal. Appl. 289 (2004) 446--455.


  \bibitem{B.Simon} B. Simon, Notes on infinite determinants of Hilbert space operators, Adv. Math., 24 (1977), 244--273. 
	
	\bibitem{Simon}	 B. Simon, Trace ideals and their applications, Cambridge University Press, 1979.



 \bibitem{Zhou2} G.-X. Tian, T.-Z. Huang and B. Zhou, A note on sum of powers of the Laplacian eigenvalues of bipartite graphs, Linear Algebra Appl., 430 (2009) 2503--2510.

  \bibitem{Weyl} H. Weyl, Inequalities between the two kinds of eigenvalues of a linear transformation, Proc. Nat. Acad. Sci. U.S.A. 35 (1949), 408--411.
	
	\bibitem{Zamani} A. Zamani and M.S. Moslehian, Exact and approximate operator parallelism, Can. Math. Bull. 58 (2015), no. 1, 207–224. 

  \bibitem{Zhou} B. Zhou, On sum of powers of the Laplacian eigenvalues of graphs, Linear Algebra Appl. 429 (2008) 2239–2246.
    
	
\end{thebibliography}

\end{document}